\documentclass[11pt]{article}

\usepackage{amsmath}
\usepackage{amssymb}
\usepackage{amsthm}
\usepackage{dsfont}
\usepackage{mathrsfs}
\usepackage{stmaryrd}
\usepackage[all]{xy}
\usepackage[mathcal]{eucal}
\usepackage{verbatim}  %
\usepackage{fullpage}  %
\usepackage{hyperref}
\usepackage{enumerate}
\usepackage{graphicx}
\usepackage{float}
\usepackage{tikz}
\usetikzlibrary{cd}
\usetikzlibrary{decorations.markings}
\usepackage{physics}

\usepackage[doi=false,isbn=false,maxcitenames=5,maxnames=4,firstinits=true]{biblatex}
\usepackage{breakurl}

\newcommand{\bbZ}{\mathbb{Z}}

\renewcommand{\phi}{\varphi}

\newtheorem{theorem}{Theorem}[section]
\newtheorem{proposition}[theorem]{Proposition}

\newtheorem*{lemma*}{Lemma}

\theoremstyle{definition}

\newtheorem*{example}{Example}

\numberwithin{equation}{subsection}

\newcommand{\set}[1]{\left\lbrace #1 \right\rbrace}

\DeclareMathOperator{\ord}{ord}

\newcommand{\Addresses}{{%
  \bigskip
  \footnotesize

  \noindent S.~Dembner, \textsc{Department of Mathematics, University of Chicago, Chicago, IL 60637}\par\nopagebreak
  \textit{Email address}: \texttt{sdembner@uchicago.edu}

  \medskip

  \noindent V.~Jain, \textsc{Department of Mathematics, Massachusetts Institute of Technology, Cambridge, MA 02139}\par\nopagebreak
  \textit{Email address}: \texttt{vanshika@mit.edu}

}}

\begin{document}

\pagestyle{plain}

\author{Spencer Dembner\thanks{Both authors were supported by the NSF (DMS-2002265), the NSA (H98230-20-1-0012), the Templeton World Charity Foundation, and the Thomas Jefferson Fund at the University of Virginia.},\, Vanshika Jain\footnotemark[1]}
\title{A Note on Congruences for Weakly Holomorphic Modular Forms}
\date{\today}

\maketitle

\begin{abstract}

Let $O_L$ be the ring of integers of a number field $L$. Write $q = e^{2 \pi i z}$, and suppose that
$$ f(z) = \sum_{n \gg - \infty} a_f(n) q^n \in M_{k}^{!}(\operatorname{SL}_2(\bbZ)) \cap O_L[[q]] $$
is a weakly holomorphic modular form of even weight $k \leq 2$. We answer a question of Ono by showing that if $p \geq 5$ is prime and $ 2-k = r(p-1) + 2 p^t$ for some $r \geq 0$ and $t > 0$, then $a_f(p^t) \equiv 0 \pmod p$. For $p = 2,3,$ we show the same result, under the condition that $2 - k - 2 p^t$ is even and at least $4$. This represents the ``missing case" of Theorem 2.5 from \cite{Seokho2016note}.
\end{abstract}

\section{Introduction}

For any even $k \geq 4$, let $M_k$ denote the space of holomorphic modular forms of level $1$ and weight $k$. A meromorphic modular form $f$ is called \emph{weakly holomorphic} if it is holomorphic on the upper half plane but may have a pole at infinity. For any even $k$, let $M_k^{!}$ denote the space of weakly holomorphic modular forms of level $1$ and weight $k$. For any $f \in M_k^!$, we write its Fourier expansion in terms of $q = e^{2 \pi i z}$ as

$$ f(z) = \sum_{n \gg - \infty} a_f(n) q^n.$$

Suppose that $f \in M_k$ is a normalized cuspidal eigenform. We say that the prime $p$ is \emph{non-ordinary} for $f$ if $p$ divides $a_f(p)$, and otherwise we say $p$ is \emph{ordinary}. Non-ordinary primes are generally expected to be rare, but little is definitively known. Elkies \cite{Elkies1987existence} showed that for any weight $2$ newform $f$, there are infinitely many primes $p$ which are non-ordinary. However, aside from forms of weight $2$, it is not known for any normalized eigenform without complex multiplication either that there are infinitely many ordinary primes, or that there are infinitely many non-ordinary primes.

In \cite{Choie2005linear}, Choie, Kohnen, and Ono show that for certain weights $k$ and primes $p$, $p$ is non-ordinary for all normalized eigenforms of level $1$ and weight $k$. Specifically, they show that if $\delta(k) \in \set{4,6,8,10,14}$, if $k \equiv \delta(k) \pmod{12}$, and if $p$ is a prime for which $k \equiv \delta(k) \pmod{p-1}$, then for all primes $\mathfrak{p} \subset O_L$ above $p$, we have
$$ a_f(p) \equiv 0 \pmod{\mathfrak{p}}. $$
Applying this result appropriately, one recovers many known cases of non-ordinary primes. For instance, one can show that $2,3,5,$ and $7$ are non-ordinary for $\Delta$, capturing all but one non-ordinary prime up to $10^6$. 

Jin, Ma, and Ono generalized this result in \cite{Seokho2016note}. This allowed them to show that given any finite set of primes $S$, there are infinitely many level $1$ normalized Hecke eigenforms $f$ such that every $p \in S$ is non-ordinary for $f$.

Beyond Hecke eigenforms, the authors proved a congruence for the coefficients of weakly holomorphic modular forms. Let $p \geq 5$ be prime. Suppose that $f = \sum_{ n \gg - \infty} a_f(n) q^n \in M_k^{!} \cap O_L[[q]]$, where $k$ is even, that we have 
$$ 2-k = r(p-1) + s p^t $$
for some $s \neq 2$, and that $\ord_{\infty}(f) > - p^u$, where $u \leq t$. Then they show that for any integer $v$ with $u \leq v \leq t$, we have 
$$ a_f(p^v) \equiv a_f(0) \equiv 0 \pmod p. $$

Ono \cite{Onoprivatecommunication} asked whether a similar result holds for $s = 2$ as well. In this note, we answer this question by proving the following.

\begin{theorem}\label{thm: Main theorem}
Suppose that $f = \sum_{n \gg - \infty} a_f(n) q^n \in M_k^{!} \cap O_L[[q]]$, where $k \in 2 \bbZ$ and $O_L$ is the ring of algebraic integers of a number field $L$. Let $p$ be a prime, and suppose that $\ord_{\infty}(f) > - p^t$. Then the following are true.

\begin{enumerate}
    \item Suppose that $p \geq 5$, and that $k \leq 2$, $r \in \bbZ_{\geq 0}$ and $t \in \bbZ_{>0}$ are integers for which
$$ 2-k = r(p-1) + 2 p^t. $$
Then we have
$$ a_f(p^t) \equiv 0 \pmod p. $$
\item Suppose that $p = 2,3,$ and that $k + 2 p^t + m = 2$, where $m \geq 4$ is even. Then we have
$$ a_f(p^t) \equiv 0 \pmod p. $$
\end{enumerate}

\end{theorem}

The proof is given in Section \ref{Sec: The Proof}. The key idea is that the forms in question are related via congruences to weight $2$ forms which arise as derivatives of modular functions.

\begin{example}
In the case $s = 2$, it is not true in general that $a_f(0) \equiv 0 \pmod p$. For instance, take $p = 5, t = 1$, and $k = -8$. The form
$$f = \frac{E_4}{\Delta} = q^{-1} + 264 + \cdots +  126745880q^5 + \cdots$$
is weakly holomorphic of weight $k$. We have $a_f(5) \equiv 0 \pmod 5$, but $a_f(0) \not \equiv 0 \pmod 5$.
\end{example}

\section{The Proof}\label{Sec: The Proof}

In Subsection \ref{subsec: preliminaries}, we recall basic facts about the Theta operator and about congruences of modular forms. In Subsection \ref{subsec: proof of main thm}, we prove Theorem \ref{thm: Main theorem}.

\subsection{Preliminaries}\label{subsec: preliminaries}

Let $f$ in $M_k^{!}$ be a weakly holomorphic modular form. Ramanujan's Theta operator is defined by 
$$ \Theta(f) = q \frac{d}{dq}(f) = \frac{1}{2 \pi i} \cdot \frac{d f}{dz}. $$

\begin{proposition}\label{prop: vanishing of constant term}
If $f = \sum_{n \gg - \infty} a_f(n) q^n \in M_2^{!}$, then $f = \Theta(P(j))$ for some polynomial in $j$. In particular, we have $a_f(0) = 0$.
\end{proposition}

\begin{proof}
For completeness, we present the proof, which is given in \cite{Seokho2016note}. Every weakly holomorphic modular form $f$ of weight $2$ is of the form $P(j(z)) E^{14}(z) \Delta(z)^{-1}$, for some polynomial $P(x)$. We have the following two key identities:
$$ - \frac{1}{2\pi i} \frac{d}{dz}j = -\Theta(j) = \frac{E_{14}}{\Delta}, $$
$$ j^w \frac{d}{dz} j = \frac{1}{w + 1} \frac{d}{dz} j^{w+1}.$$
If $Q$ is a polynomial such that $Q'(x) = -P(x)$, it follows that we have
$$ \Theta( Q(j)) =  \frac{-1}{2 \pi i} P(j) \frac{d}{dz}j = P(j) \frac{E_{14}}{\Delta} = f.  $$
The derivative with respect to $z$ of a $q$-series has vanishing constant term, so we have $a_f(0) = 0$. \qedhere

\end{proof}

We will also need the following well-known congruences for Eisenstein series.

\begin{proposition}\label{prop: congruence for eisenstein series}
If $p \geq 5$ is prime, then as a $q$-series, $E_{p-1}(z) \equiv 1 \pmod p$. Additionally, for any $k \geq 2$, we have $E_k(z) \equiv 1 \pmod{24}$.
\end{proposition}

\begin{proof}
See Lemma 1.22, parts (1) and (2), from \cite{Ono2004web}. \qedhere
\end{proof}

\subsection{Proof of Theorem 1.1} \label{subsec: proof of main thm}

Let $g = \Theta(j) \in M_2^{!}$. We have
$$ g = - q^{-1} + a_j(1) q + 2 a_j(2) q^2 + \cdots.   $$
It follows that we have
$$ g^{p^t} \equiv \pm q^{- p^t} + O(q^{p^t}) \pmod p. $$
In the case where $p = 2, 3$, we can pick $c_1, c_2$ such that $4 c_1 + 6 c_2 = m$. We define
$$ h = \begin{cases} g^{p^t} E_{p-1}^r f & p \geq 5, \\ g^{p^t} E_4^{c_1} E_6^{c_2} f & p = 2 ,3. \end{cases} $$
In both cases, $h$ is weakly holomorphic of weight $2$. It follows by Proposition \ref{prop: vanishing of constant term} that $a_h(0) = 0$. By Proposition  \ref{prop: congruence for eisenstein series}, we have $E_{p-1}^r \equiv 1 \pmod p$ for $p \geq 5$, and $E_4^{c_1} E_6^{c_2} \equiv 1 \pmod p$ for $p = 2,3$. Since additionally $\ord_{\infty}(f) > - p^t$, we have
$$a_h(0) \equiv \pm a_f(p^t) \equiv 0 \pmod p,$$
which proves the desired claim. \qed

\printbibliography

\Addresses

\end{document}